\documentclass{amsart}  

\usepackage{amsfonts}
\usepackage{amsthm}
\usepackage{amsmath}
\usepackage{amsfonts}
\usepackage{latexsym}
\usepackage{amssymb}
\usepackage{amscd}
\usepackage[latin1]{inputenc}
\usepackage{verbatim}

\newtheorem{theorem}{Theorem}[section]
\newtheorem{lemma}[theorem]{Lemma}
\newtheorem{proposition}[theorem]{Proposition}
\newtheorem{corollary}[theorem]{Corollary}

\theoremstyle{definition}
\newtheorem{definition}[theorem]{Definition}
\newtheorem{remark}[theorem]{Remark}

\theoremstyle{remark}

\newcommand\style{\mathcal }          

\newcommand\hil{\style H}                        
\newcommand\hilk{\style K}                       
\newcommand\bh{\style B(\style H)}    

\newcommand\bof{\style B}

\newcommand\mn{{\style M}_n}

\newcommand\mpee{{\style M}_p}

\newcommand\mpeen{{\style M}_{pn}}

\newcommand\mtwo{{\style M}_2}

\newcommand\tn{{\style T}_n}            
\newcommand\tthree{{\style T}_3}

\newcommand\fn{{\mathbb F}_n}         
\newcommand\fnl{{\mathbb F}_{n-1}}
\newcommand\fni{{\mathbb F}_{\infty}}
\newcommand\ftwo{{\mathbb F}_2} 

\newcommand\osr{{\style R}}
\newcommand\oss{{\style S}}
\newcommand\ost{{\style T}}

\newcommand\kj{{\style J}}                    

\newcommand\omin{\otimes_{\rm min}}
\newcommand\omax{\otimes_{\rm max}}
\newcommand\oc{\otimes_{\rm c}}

\newcommand\coisubset{\subset_{\rm coi}}  

\newcommand\pstar{{\rm (}\mathfrak W{\rm )}}     
 
\newcommand\psn{{\rm (}\mathfrak S_n{\rm )}} 
 
\newcommand\psthree{{\rm (}\mathfrak S_3{\rm )}} 

\newcommand\csta{{\style A}}
\newcommand\cstb{{\style B}}

\newcommand\cstar{{\rm C}^*}                              

\begin{document}
 
\title[C$^*$-Algebras with WEP]{C$^*$-Algebras with the Weak Expectation Property 
and a Multivariable Analogue of Ando's Theorem on the Numerical Radius}

\author[D.~Farenick]{Douglas Farenick}
\address{Department of Mathematics and Statistics, University of Regina,
Regina, Saskatchewan S4S 0A2, Canada}
\email{douglas.farenick@uregina.ca}

\author[A.~S.~Kavruk]{Ali S.~Kavruk}
\address{Department of Mathematics, University of Houston,
Houston, Texas 77204-3476, U.S.A.}
\email{kavruk@math.uh.edu}

\author[V.~I.~Paulsen]{Vern I.~Paulsen}
\address{Department of Mathematics, University of Houston,
Houston, Texas 77204-3476, U.S.A.}
\email{vern@math.uh.edu}


\begin{abstract}
A classic theorem of T.~Ando characterises operators that have numerical radius at most one
as operators that admit a certain positive $2\times 2$ operator
matrix completion. In this paper we consider variants of Ando's theorem
in which the operators (and matrix completions) are constrained to a given C$^*$-algebra. By considering $n\times n$ matrix completions,
an extension of Ando's theorem to a multivariable setting is made. We show that the C$^*$-algebras in which these extended formulations of
Ando's theorem hold true are precisely the C$^*$-algebras with the weak expectation property (WEP). We also show that a C$^*$-subalgebra of $\bh$
has WEP if and only if whenever a certain $3\times 3$ (operator) matrix completion problem
can be solved in matrices over $\bh$, it can also be solved in matrices over $\csta$. This last result gives a 
characterisation of WEP that is spatial and yet is independent of the particular representation of the C$^*$-algebra. This leads to a new characterisation
of injective von Neumann algebras. We also give a new equivalent formulation of the Connes Embedding Problem as a problem concerning
$3\times3$ matrix completions.
\end{abstract}

\maketitle

\section{INTRODUCTION}

The \emph{numerical radius} of a bounded linear operator $X$ acting on a Hilbert space $\hil$  is the quantity
$w(X)$ defined by
\[ w(X) \,=\, \sup \{ |\langle X\xi,\xi \rangle| \, :\, \xi \in \hil,\, \|\xi\| =1 \}. 
\]
A classic theorem of Ando \cite{ando1973}
gives a matricial positivity characterisation of the numerical radius of an operator:
$w(X) \le 1/2$ if and only if there exist positive operators $A,B \in \bh$ such that $A+B=I\in\bh$ and
$\left[ \begin{array}{cc} A& X \\ X^* & B \end{array} \right] $
is a positive operator on $\hil\oplus\hil$.
With a little rescaling, one can see that Ando's theorem is equivalent
to the statement that $w(X) < 1/2$ if and only if there exist positive
invertible operators $A,B \in \bh$ with $A+B=I$ such that the $2\times 2$ operator matrix above
is positive and invertible.

It is natural to wonder if this theorem remains true when $\bh$ is replaced by an arbitrary unital C*-algebra. The answer is yes,
as long as one requires that the inequality be strict. Thus, we have the following minor improvement of Ando's theorem,
where, in its formulation below, $\csta_+$ and $\csta_+^{-1}$ denote the set of positive elements and the group of positive invertible elements
of a unital C$^*$-algebra $\csta$ respectively, and where the numerical radius $w(x)$ of $x\in\csta$
is the maximum value of $|s(x)|$ as $s$ ranges through all states of $\csta$.

\begin{theorem}\label{ando thm} Assume that $\csta$ is any unital C$^*$-algebra. The following statements
are equivalent for $x\in\csta$:
\begin{enumerate}
\item $w(x) < \frac{1}{2}$; 
\item for every unitary $v\in\cstb$ in every unital C$^*$-algebra $\cstb$, the element $x\otimes v\in \csta\omin\cstb$ satisfies $w(x\otimes v) < \frac{1}{2}$;
\item for every unital C$^*$-algebra $\cstb$,
\[
1_\csta\otimes1_\cstb+x\otimes v + x^*\otimes v^*\in\left(\csta\omin\cstb\right)_+^{-1}
\]
for every unitary $v\in\cstb$;
\item $1_\csta\otimes1_\cstb+x\otimes u + x^*\otimes u^*\in\left(\csta\omin\cstb\right)_+^{-1}$, where $\cstb$ is the universal C$^*$-algebra generated by a 
(universal) unitary $u$;
\item $1+zx+\overline{z}x^*\in\csta_+^{-1}$ for every $z\in\mathbb T$;
\item there exist $a,b\in\csta_+^{-1}$ such that
\[
\left[ \begin{array}{cc} a & x \\ x^* & b\end{array}\right] \in \mtwo(\csta)_+\,
\] 
with $a+b= 1$.
\end{enumerate}
\end{theorem}

Surprisingly, this slight extension of Ando's Theorem is logically equivalent to the assertion that $C(\mathbb T)$ is a nuclear C$^*$-algebra
(Theorem \ref{interesting}(1)).

It is also natural to wonder if it is possible to formulate Ando's
Theorem for a greater number of variables. Since the numerical radius
of an operator remains unchanged when one tensors with a unitary, we
have been led to the following definition.

\begin{definition} The \emph{free joint numerical radius} $w(X_1, \ldots,X_n)$ of $n$ operators $X_1,\dots,X_n\in \bh$
is
\[w(X_1, \ldots,X_n) = \sup \{ w(X_1\otimes U_1 +\dots + X_n\otimes
U_n) \}\,,\]
where the supremum is taken over every Hilbert space $\hilk$, every
choice of $n$ unitaries $U_1,\dots, U_n \in \bof(\hilk)$, and the tensor
product is spatial.
\end{definition}

We obtain a characterisation of $n$ tuples of operators with
$w(X_1,\ldots,X_n) < 1/2$ in terms of matrix positivity (Theorem~\ref{multando}),
which is a natural extension of Ando's Theorem.
What is perhaps most surprising is that when one asks whether the entries of this operator matrix can be chosen from a given C*-algebra 
$\csta$ containing $X_1, \dots, X_n$,
we find that the extension of Ando's result to two or more variables holds if and only if the C$^*$-algebra $\csta$ has the weak expectation property (WEP).

In this manner, we obtain a characterisation of the WEP property for a 
C$^*$-subalgebra of $\bh$ that is in the spirit of completion problems. 
That is, we prove in Theorem \ref{wep3x3-1} that a C$^*$-subalgebra of 
$\bh$ has WEP if and only if whenever a certain $3 \times 3$ operator 
completion problem can be solved in $\bh$,
then it can be solved with entries from the C$^*$-algebra.

This characterisation of WEP is independent of the particular faithful 
representation of the C$^*$-algebra on a Hilbert space. In this regard Theorem \ref{wep3x3-1} departs
from the original definition of WEP, which requires that {\it every} faithful 
representation of the given C$^*$-algebra 
admits a weak expectation into its double commutant. This requirement in the original definition of WEP is
crucial, as every C$^*$-algebra has at 
least one faithful representation which has a weak expectation into its double commutant.

Since a von~Neumann algebra has WEP if and only if it is injective, our 
results also give a characterisation of injective von~Neumann algebras in terms of 
our $3 \times 3$ completion property (Corollary \ref{inj1}). In particular, this shows 
that the operators that solve our $3 \times 3$ completion problem for a given unital 
C$^*$-subalgebra $\csta \subset \bh$ without WEP will generally not
even be found within the double commutant of $\csta$.

Finally, by Kirchberg's results \cite{kirchberg1993}, we find in Theorem \ref{cep}
that Connes' Embedding Problem is equivalent to whether or not $\cstar(\ftwo)$ has our $3 \times 3$ completion property. 
This might represent some progress on this conjecture, since, as noted already, it is sufficient to check our property for any chosen faithful representation.

\section{OPERATOR SYSTEMS}

If $\psi:\osr\rightarrow\ost$ is a linear map of operator systems, then for each $p\in\mathbb N$ the linear map
$\psi^{(p)}:\mpee(\osr)\rightarrow\mpee(\ost)$ is defined by $\psi^{(p)}\left([x_{ij}]_{i,j}\right)=\left[\psi(x_{ij})\right]_{ij}$. The positive matricial cones
of an operator system $\osr$ are denoted by $\mpee(\osr)_+$, and the order unit of $\osr$ is denoted by $1_\osr$ or, if  there is little chance of ambiguity,
by $1$. 

Two (classes of) operator systems have a prominent role in what follows. The first, $\tn$, is an operator subsystem of
the $n\times n$ complex matrix algebra $\mn$. If $\{e_{ij}\}_{1\leq i,j\leq n}$ denotes the set of standard matrix
units of $\mn$, then
\[
\tn\,=\,\mbox{Span}\{e_{ij}\,:\,|i-j|\leq 1\} 
\]
is the operator system of tridiagonal matrices.

The second operator system of interest is denoted by $\oss_n$.
Assume that $\fni$ is the free group on countably many generators $u_1,u_2,\dots$ and let $\fn\subset\fni$
be the free group on $n$ generators $u_1,\dots,u_n$. In the group C$^*$-algebra $\cstar(\fni)$ these generators $u_j$ are
(universal) unitaries. For a fixed $n\in\mathbb N$, consider the operator subsystem $\oss_n\subset\cstar(\fn)$ defined by
\[
\oss_n\,=\,\mbox{Span}\{u_{-n}, u_{-n+1},\dots, u_{-1}, u_0, u_1, \dots, u_{n-1}, u_n\},
\]
where
$u_0=1$ and $u_{-k}=u_k^*$ for $1\leq k\leq n$.

The relationship between the two types of operator systems described above is as follows. Let $n\geq2$ and let
$\phi:\tn\rightarrow\oss_{n-1}$ be the linear map defined by
\[
\phi(e_{ii}) = \frac{1}{n}u_0,\;  \phi(e_{j,j+1}) = \frac{1}{n}u_j, 
\;\mbox{ and }\; \phi(e_{j+1,j}) = \frac{1}{n}u_{-j}
\]
for $i=1,\dots,n$ and $j=1,\dots,n-1$.
The map $\phi$ is ucp  \cite[Theorem~4.2]{farenick--paulsen2011}, but the most important fact about $\phi$ is
that it is a \emph{complete quotient map}, which we explain below. 

Recall from \cite{kavruk--paulsen--todorov--tomforde2010}
that if $\psi:\osr\rightarrow\ost$ is a surjective
completely positive linear map of operator systems with kernel $\kj$, then
there is an operator system structure on the quotient vector space $\osr/\kj$ and there exists a completely 
positive linear isomorphism $\dot{\psi}:\osr/\kj\rightarrow\ost$ such that $\psi=\dot{\psi}\circ q$, where 
$q:\osr\rightarrow\osr/\kj$ is the canonical quotient map. If the completely positive linear isomorphism $\dot{\psi}$ is 
a complete order isomorphism (that is, if $\dot{\psi}^{-1}$ is completely positive), then $\psi$ is said to be a complete quotient map.
  
Returning to the map $\phi:\tn\rightarrow\oss_{n-1}$ above, we have from \cite[Theorem~4.2]{farenick--paulsen2011} that
$\phi$ is a complete quotient map; that is, the operator systems
$\tn/\kj,$ where $\kj = \ker(\phi)$, and $\oss_{n-1}$ are completely
order isomorphic. The importance of this fact concerning $\phi$ is that strictly positive elements in the matrix space $\mpee(\oss_{n-1})$
lift back to strictly positive elements in $\mpee(\tn)$ for every $p\in\mathbb N$ (Proposition \ref{sp}). 
This fact, when applied while tensoring these operator systems
with a C$^*$-algebra, is at the heart of our matrix completion perspective.

The fundamental results for a theory of operator system tensor products are developed
in \cite{kavruk--paulsen--todorov--tomforde2011,kavruk--paulsen--todorov--tomforde2010}. An operator system tensor product $\osr\otimes_\tau\ost$
is an operator system structure on the algebraic tensor product $\osr\otimes\ost$ satisfying a set of natural axioms.
Given two operator system tensor product structures $\osr\otimes_{\tau_1}\ost$ and  $\osr\otimes_{\tau_2}\ost$ on $\osr\otimes\ost$,
we of course have equality of $\osr\otimes_{\tau_1}\ost$ and
$\osr\otimes_{\tau_2}\ost$ as sets; however, we use the
notation
$\osr\otimes_{\tau_1}\ost_1 \,=\, \osr\otimes_{\tau_2}\ost_1$
to indicate that the identity map is a complete
order isomorphism. 
Given two inclusions, $\osr_1 \subseteq \osr_2$ and $\ost_1
\subseteq \ost_2$ and operator system tensor product structures
$\osr_1\otimes_{\tau_1} \ost_1$ and $\osr_2\otimes_{\tau_2}\ost_2,$ we use the notation
 $\osr_1\otimes_{\tau_1}\ost_1 \,\coisubset\,
 \osr_2\otimes_{\tau_2}\ost_2$ to indicate that the tensor product of
 the two inclusion maps is a complete order isomorphism onto its range.

Of particular interest here are the tensor products 
$\omin$, $\oc$, and $\omax$
\cite{kavruk--paulsen--todorov--tomforde2011}. In this case, we have that
the tensor products of the identity maps are ucp as maps from
$\osr\omax\ost$ to $\osr\oc\ost$ and from  $\osr\oc\ost$ to $\osr\omin\ost$
for all operator systems $\osr$, $\ost$. Indeed, the matricial cones associated with $\osr\omax\ost$ lie within
the matricial cones of any operator system structure $\osr\otimes_\tau\ost$, while the matricial cones of
any $\osr\otimes_\tau\ost$ are contained in the corresponding matricial cones of $\osr\omin\ost$.

We identify $\mpee\left(\osr\omin\ost\right)$ with $\mpee(\osr)\omin\ost$, for any 
operator systems $\osr$ and $\ost$.

When considering operator system tensor products of unital C$^*$-algebras $\csta$ and $\cstb$, 
the identity map of the operator system $\csta\omin\cstb$ into the C$^*$-algebra tensor product $\csta\omin\cstb$ is
completely positive \cite[Corollary 4.10]{kavruk--paulsen--todorov--tomforde2011}. The analogous statement for the $\omax$ tensor product also holds
\cite[Theorem 5.12]{kavruk--paulsen--todorov--tomforde2011}.
Because of the norm--order duality in operator systems \cite[\S4]{choi--effros1977}, \cite[Proposition 13.3]{Paulsen-book}, 
we have that $\csta\omin\cstb=\csta\omax\cstb$ as
operator systems if and only if the algebraic tensor product
$\csta\otimes\cstb$ has a unique C$^*$-norm if and only if
$\csta\omin\cstb=\csta\omax\cstb$ as C$^*$-algebras. 
For this reason, the equation  $\csta\omin\cstb=\csta\omax\cstb$ can be treated unambiguously as
a statement about operator systems or as a statement about C$^*$-algebras.

The following lifting property \cite{farenick--paulsen2011} for positive matrices over operator systems is a key feature of our approach.
An operator system $\osr$ is said to have \emph{property $\psn$} 
for a fixed $n\in\mathbb N$ if, for every $p\in\mathbb N$, every $\varepsilon>0$, and every positive
\[
\sum_{i=1-n}^{n-1}S_{i}\otimes u_i \,\in\, \mpee(\osr)\omin\cstar(\fnl)\,,
\]
there exist  $R_{ij}^\varepsilon\in  \mpee(\osr)$, for $1\leq i,j \leq n$, such that
\begin{enumerate}
\item $R_\varepsilon=[R_{ij}^\varepsilon]_{1\leq i,j\leq n}$ is positive in $\mn(\mpee(\osr))$,
\item $R_{ij}^\varepsilon=0$ for all $|i-j|\geq2$, $R_{i,i+1}^\varepsilon=S_{i}$, and $R_{i+1,i}^\varepsilon=S_{-i}$ for all $i$, and
\item $\displaystyle\sum_{i=1}^n R_{ii}^\varepsilon \,=\,S_0+\varepsilon 1_{\mpee(\osr)}$.
\end{enumerate}

To conclude this section, we connect the notions described above using the operator systems $\tn$ and $\oss_{n-1}$,
the linear map $\phi$ that links them, and property $\psn$.
 
\begin{theorem}\label{fp thm} {\rm (\cite{farenick--paulsen2011})} Assume that $n\geq2$ and $\osr$ is an arbitrary operator system.
\begin{enumerate}
\item The map ${\rm id}_\osr\otimes\phi:\osr\omax\tn\rightarrow\osr\omax\oss_{n-1}$ is a complete quotient map.
\item The map ${\rm id}_\osr\otimes\phi:\osr\omin\tn\rightarrow\osr\omin\oss_{n-1}$ is a complete quotient map if and only if $\osr$ has property $\psn$.
\end{enumerate}
\end{theorem}

\section{PROPERTY $\psn$ AND A MULTIVARIABLE ANALOGUE OF ANDO'S THEOREM}

In this section we use the fact that $\bh$ has property $\psn$ 
\cite[Theorem 6.2]{farenick--paulsen2011} to derive a multivariable
analogue of Ando's theorem.

\begin{definition} If $\osr$ is an operator system with order unit $1_\osr$, then $s\in\osr$ is 
\emph{strictly positive} if there is a real number $\delta>0$ such that $s\geq\delta1_\osr$.
\end{definition}

This terminology is not entirely standard.
It is not hard to see that an element $s \in \osr$ is strictly
positive if and only if for every unital C*-algebra $\csta$ and every
ucp map $\psi: \osr \to \csta$ we have that $\psi(s)$ is positive and
invertible. Thus, in our terminology $P \in \bh$ is strictly positive if
and only if $P$ is positive and invertible.

\begin{proposition}\label{sp} The following statements are equivalent for a  
ucp surjection $\psi:\osr\rightarrow\ost$ of operator systems:
\begin{enumerate}
\item\label{sp-1} for every $p\in\mathbb N$ and every strictly 
positive $y\in\mpee(\ost)$ there is a strictly positive $x\in\mpee(\osr)$ such that $\psi^{(p)}(x)=y$;
\item\label{sp-2} $\psi$ is a complete quotient map.
\end{enumerate}
\end{proposition}

\begin{proof} The proof in the case of general $p\in\mathbb N$ is no different than the proof in the case $p=1$, and
so we settle on this case for simplicity of notation. 

\eqref{sp-1}$\Rightarrow$\eqref{sp-2}. Assume that $y\in\ost$ is strictly positive. 
Let $\dot{x}\in(\osr/\ker\psi)$ be the unique preimage of $y$ under $\dot{\psi}$; note that $\psi(x)=y$. We aim
to show that $\dot{x}$ is positive. By definition of positivity in 
the quotient \cite[\S3]{kavruk--paulsen--todorov--tomforde2010}, we are to show that
for every $\varepsilon>0$ there is a $k_\varepsilon\in\ker\psi$ 
such that $\varepsilon1_\osr+x+k_\varepsilon\in\osr_+$. Fix $\varepsilon>0$. Because $y+\varepsilon1_\ost$
is strictly positive, there is a strictly positive $x_\varepsilon\in\osr$ such that $\psi(x_\varepsilon)=y+\varepsilon1_\ost$.
Thus, if $k_\varepsilon=x_\varepsilon-(x+\varepsilon1_\osr)$, then $k_\varepsilon\in\ker\psi$ and
$\varepsilon1_\osr+x+k_\varepsilon=x_\varepsilon$ is stricly positive in $\osr$. Hence, $\dot{x}$ is positive.

\eqref{sp-2}$\Rightarrow$\eqref{sp-1}. 
Assume that $a\in\ost$ is strictly positive: $a\geq\delta1_\ost$ for some $\delta>0$.
Thus, $y=a-\delta1_\ost\in\ost_+$ and so $z=y+\frac{\delta}{2}1_\ost\in\ost_+$. By hypothesis, $\dot{\psi}$ is
a complete order isomorphism and so $z=\dot{\psi}(\dot{h})$ for some positive $\dot{h}\in (\osr/\ker\psi)$.
By definition of positivity in the quotient, there
is a $k\in\ker\psi$ such that $\frac{\delta}{4}+h+k\in\osr_+$. Let $q=\frac{\delta}{2}+h+k$; thus,
$q\geq\frac{\delta}{4}1_\osr$
and $\psi(q)=\dot{\psi}(\dot{q})=\frac{\delta}{2}1_\ost+z=a$.
\end{proof}

Note that for the implication \eqref{sp-1}$\Rightarrow$\eqref{sp-2} in Proposition \ref{sp} above, 
it is enough to check that strictly positive elements have positive lifts. Also, by adding and subtracting fractions of $\delta$, a 
positive lift of a strictly positive element can be replaced by a strictly positive lift.

We now apply Proposition \ref{sp} to obtain a characterisation of property $\psn$ in terms of a strictly positive matrix completion
condition: namely, in the formulation below, for every $p\in\mathbb N$, variables $x_1,\dots,x_{n-1}\in\mpee(\csta)$ 
determine a strictly positive element \eqref{sp in target space}
of $\mpee(\csta)\omin\cstar(\fnl)$ if and only if the partially specified hermitian tridiagonal matrix $\tilde T$ with superdiagonal given by $x_1,\dots,x_{n-1}$
and diagonal entries unspecified
can be completed to a strictly positive matrix $T$ in $\mpeen(\csta)$ with diagonal entries that sum to the identity of $\mpee(\csta)$.

\begin{theorem}\label{lifting to matrices} For a fixed $n\geq2$, the following statements are equivalent for a unital C$^*$-algebra $\csta$:
\begin{enumerate}
\item\label{lift-1} $\csta$ has property $\psn$;
\item\label{lift-2} for every $p\in\mathbb N$ and every $x_1,\dots,x_{n-1}\in\mpee(\csta)$ for which
\begin{equation} \label{sp in target space}
1\otimes 1 \,+\,\sum_{j=1}^{n-1} (x_j\otimes u_j) \,+\,
\sum_{j=1}^{n-1} (x_j^*\otimes u_j^*)  \tag{3.3.1}
\end{equation}
is strictly positive in $\mpee(\csta)\omin\cstar(\fnl)$, the matrix
\begin{equation}\label{pullback}
\left[ \begin{array}{ccccc} a_1 & x_1 & 0 & \cdots & 0 \\ x_1^* & a_2 & x_2 & & \vdots \\
0& x_2^* & \ddots & \ddots & 0 \\ \vdots & & \ddots & a_{n-1} & x_{n-1} \\ 
0 & \cdots & 0 & x_{n-1}^* & a_n \end{array}
\right] \tag{3.3.2}
\end{equation}
is strictly positive in $\mpeen(\csta)$ for some 
$a_1,\dots,a_{n-1}\in\mpee(\csta)$ such that
$a_1+a_2+\dots+a_n=1\in\mpee(\csta)$.
\end{enumerate}
\end{theorem}

\begin{proof} 
\eqref{lift-1}$\Rightarrow$\eqref{lift-2}. Because $\csta$ has property $\psn$,
Theorem \ref{fp thm} implies that
the linear map
${\rm id}_{\mpee(\csta)}\otimes\phi:\mpee( \csta)\omin\tn\rightarrow\mpee(\csta)\omin\oss_{n-1}$ is a complete quotient map.
Hence, if 
$z=1\otimes 1 \,+\,\sum_{j=1}^{n-1} (x_j\otimes u_j) \,+\,
\sum_{j=1}^{n-1} (x_j^*\otimes u_{-j})$ is strictly positive in $\csta\omin\cstar(\fnl)$, then there is a
strictly positive lift $y\in\mpee(\csta)\omin\tn$ with 
$[{\rm id}_{\mpee(\csta)}\otimes\phi](y)=z$, by Proposition \ref{sp}. The element $y$ necessarily has the form
\[
y\,=\,\sum_{j=1}h_j\otimes e_{jj}\,+\,\sum_{i=1}^{n-1} k_i\otimes e_{i,i+1} \,+\, \sum_{i=1}^{n-1} k_i^*\otimes e_{i+1,i},
\]
for some $k_1,\dots,k_{n-1}\in\mpee(\csta)$ and (strictly) positive $h_1,\dots, h_n\in\mpee(\csta)$. Let $\delta>0$ be such that
$\delta 1\leq y$. Thus, 
\[
\begin{array}{rcl}
\delta(1\otimes 1)\,\leq\,z &=&
\displaystyle\sum_{j=1}h_j\otimes\left(\frac{1}{n}1\right)\,+\,
\displaystyle\sum_{i=1}^{n-1} k_i\otimes\left(\frac{1}{n}u_i\right) \,+\, 
\displaystyle\sum_{i=1}^{n-1} k_i^*\otimes \left(\frac{1}{n}u_{i}\right) 
\\ && \\
&=&
\left(\displaystyle\sum_{j=1}\frac{1}{n}h_j\right)\otimes 1\,+\,
\displaystyle\sum_{i=1}^{n-1} \left(\frac{1}{n}k_i\right)\otimes  u_i  \,+\, 
\displaystyle\sum_{i=1}^{n-1} \left(\frac{1}{n}k_i^*\right)\otimes u_{i}\,.
\end{array}
\]
The linear independence of $u_{1-n},\dots,u_0,\dots, u_{n-1}$ implies that $z$ is the image of $y$ if and only
if $k_i=nx_i$ for $1\leq i\le n$
and $\displaystyle\sum_{j=1}^n\frac{1}{n}h_j=1$.
Now since $y$ is strictly positive, so is $\frac{1}{n}y$. Therefore, if $a_j=\frac{1}{n}h_j$ for each $j$, 
then $a_1+\cdots+a_n=1$ and
the strictly positive element $\frac{1}{n}y$ is given by the matrix  \eqref{pullback}.

\eqref{lift-2}$\Rightarrow$\eqref{lift-1}.  By Theorem \ref{fp thm} and Proposition \ref{sp}, to show that $\csta$ has property
$\psn$ it is enough to show that for every strictly positive $g\in \mpee(\csta)\omin\oss_{n-1}$ 
there exists a strictly positive $h\in\mpee(\csta)\omin\tn$ such that 
$[{\rm id}_{\mpee(\csta)}\otimes\phi](h)=g$. Assuming that $g\in \mpee(\csta)\omin\oss_{n-1}$
is strictly positive, there exist $\delta>0$ and $x_k\in\mpee(\csta)$ such that
\[
\delta(1\otimes1)\,\leq\,g\,=\, x_0\otimes 1 \,+\,\sum_{j=1}^{n-1} (x_j\otimes u_j) \,+\,
\sum_{j=1}^{n-1} (x_j^*\otimes u_{-j})\,.
\]
If $\alpha$ is an arbitrary state on $\mpee(\csta)$ and $\beta$ is a state on $\cstar(\fnl)$ such that 
$\beta(u_k)=0$ for all $k\neq0$, then $\delta\leq\alpha\otimes\beta(g)=\alpha(x_0)$. Thus, $x_0$
is strictly positive. Thus, $z=(x_0^{-1/2}\otimes1)g(x_0^{-1/2}\otimes1)$ is strictly positive
and has the form \eqref{sp in target space}; therefore, by hypothesis, $z$ has a strictly positive lift
some strictly positive $y\in\mpee(\csta)\otimes\tn$. The proof of the implication \eqref{lift-1}$\Rightarrow$\eqref{lift-2}
shows how the entries of the matrix $y$ are determined by those of the matrix $z$. 
But since $x_0^{-1/2}x_ix_0^{-1/2}$ are the tensor factors of $u_i$ in the expansion of $z$ as
a sum of elementary tensors, 
one sees immediately (using the computations
in \eqref{lift-1}$\Rightarrow$\eqref{lift-2}) that $[{\rm id}_{\mpee(\csta)}\otimes\phi](h)=g$ for some
strictly positive $h\in\mpee(\csta)\omin\tn$.
\end{proof}
 
We can now state our multivariable analogue of Ando's theorem.

 \begin{theorem}\label{multando} Let $X_1, \dots, X_{n-1} \in
   \bh$. Then $w(X_1, \ldots, X_{n-1}) < 1/2$ if and only if there
   exist $A_1, \dots, A_n \in \bh_+^{-1}$ with $A_1+ \cdots + A_n =I$
   such that
   \begin{equation}\label{matrix}
\left[ \begin{array}{ccccc} A_1 & X_1 & 0 & \cdots & 0 \\ X_1^* & A_2 & X_2 & & \vdots \\
0& X_2^* & \ddots & \ddots & 0 \\ \vdots & & \ddots & A_{n-1} & X_{n-1} \\ 
0 & \cdots & 0 & X_{n-1}^* & A_n \end{array} \right]   \tag{3.4.1}
\end{equation}
is positive and invertible in $M_n(\bh)= \bof(\hil^{(n)}).$
\end{theorem}

\begin{proof} Since $\bh$ is injective it
has WEP and thus, by \cite[Theorem~6.2]{farenick--paulsen2011}, $\bh$ has property $\psn$ for all
$n$. (Alternatively, \cite[Proposition~3.5]{farenick--paulsen2011}
gives a direct proof that $\bh$ has the lifting property $\pstar$,
which is easily seen to imply the lifting property $\psn$.)

We have that $w(X_1 \otimes U_1 + \cdots + X_{n-1} \otimes U_{n-1}) <
1/2$ for all unitaries if and only if $1 \otimes 1 + \sum_{j=1}^{n-1}
(X_j \otimes u_j) + \sum_{j=1}^{n-1} (X_j \otimes u_j)^*$ is strictly
positive in $\bh \omin \cstar(\fnl)$.  Thus, by applying
Theorem~\ref{lifting to matrices} we have the desired lifting.

Conversely, assume that, for a given sequence $X_1, \dots, X_{n-1}\in\bh$, there exist operators
$A_1, \dots, A_n$ satisfying the conditions above.  Tensoring the complete
quotient map $\phi: \tn \to \oss_{n-1}$ with the identity map on
$\bh$ yields a ucp map ${\rm id}_{\bh}\otimes\phi: \bh\omin\tn \to
\bh\omin\oss_{n-1}$. The image of the operator matrix \eqref{matrix} under
this map is
\[
\left(\sum_{j=1}^nA_j\right)
 \otimes 1 + \sum_{j=1}^{n-1}\left(X_j \otimes u_j \,+\,
 X_j^* \otimes u_j^*\right) = I \otimes 1 + \sum_{j=1}^{n-1}\left(X_j \otimes u_j \,+\,
 X_j^* \otimes u_j^*\right), 
 \] which will be strictly
positive.
Hence, $w(X_1 \otimes U_1 + \cdots + X_{n-1} \otimes U_{n-i}) < 1/2,$
for any set of unitaries, $U_1, \dots, U_{n-1}$.
\end{proof}

The definition of the free joint numerical radius involves a supremum,
but Theorem \ref{multando} allows us to also characterise it as an infimum.

\begin{corollary} If $X_1, \ldots, X_{n-1} \in \bh$, then
\[ w(X_1, \ldots , X_{n-1}) \,=\, \inf \{ \|A_1+ \cdots +A_n\| \}, \]
where the infimum is taken over all sets of operators $A_1, \ldots,
A_n \in \bh$ for which the operator matrix \eqref{matrix} is positive in
$\bof(\hil^{(n)}).$
\end{corollary}

When $n=2$ we recover the original version of Ando's theorem. Thus, we
see that Ando's theorem can be derived as a consequence of the fact that $\tn \to
\oss_{n-1}$ is a complete quotient map and that $\bh$ has the lifting property $\psn$.

\
\section{PROPERTY $\psn$ FOR C$^*$-ALGEBRAS}

If, in the universal representation $\csta\subset\bof(\hil_u)$ of a unital C$^*$-algebra $\csta$, there exists
a ucp map $\Phi:\bof(\hil_u)\rightarrow \csta^{**}$ such that $\Phi(a)=a$ for all $a\in A\subset A^{**}\subset\bof(\hil_u)$, then
$\csta$ is said to have the \emph{weak expectation property} (WEP). Equivalently, $\csta$ has WEP if for every
faithful representation $\pi:\csta\rightarrow\bof(\hil_\pi)$ there is a ucp $\Phi_\pi:\bof(\hil_\pi)\rightarrow \pi(\csta)^{''}$ such that $\Phi(\pi(a))=\pi(a)$ for all 
$a\in \pi(A)\subset \pi(A)^{''}\subset\bof(\hil_\pi)$.

An important feature of the C$^*$-algebras of free groups is that they may be used to detect C$^*$-algebras with WEP. This
is achieved through an important theorem of Kirchberg \cite[Proposition 1.1(iii)]{kirchberg1993}: a C$^*$-algebra has WEP if and only if 
there is a unique C$^*$-norm on the algebraic tensor product $\csta\otimes\cstar(\fni)$. We prove below (in Theorem \ref{alg psn}) that individual operator systems
$\oss_n$, assuming $n$ is at least $3$,  also detect C$^*$-algebras with WEP. 

The main results of this section can be derived as consequences of
results that appear in the thesis of the second author. See especially
Section 5 of \cite{kavruk}. We give an independent derivation of these
facts below.

We begin with two preliminary lemmas.

\begin{lemma}\label{incl lemma} $\osr\oc\oss_n\coisubset\osr\oc\cstar(\fn)$ for every $n\in\mathbb N$ and every operator system $\osr$.
\end{lemma}

\begin{proof} By definition of $\oc$ \cite[\S6]{kavruk--paulsen--todorov--tomforde2011}, 
it suffices to show that, for every pair of ucp maps $\phi:\osr\rightarrow\bh$ and $\psi:\oss_n\rightarrow\bh$
with communing ranges, there is a ucp extension $\tilde\psi:\cstar(\fn)\rightarrow\bh$ of $\psi$ such that $\tilde\psi$ and $\phi$ have commuting ranges.

To this end, let $\phi$ and $\psi$ be such a pair. Dilate each contraction $\psi(u_i)$ to a unitary $w_i\in \bof(\hil\oplus\hil)$:
\[
w_i\,=\,\left[ \begin{array}{cc} \psi(u_i) & \left( 1-\psi(u_i)\psi(u_{-i})\right)^{1/2} \\ 
                                               \left( 1-\psi(u_{-i})\psi(u_{i})\right)^{1/2} & -\psi(u_{-i}) \end{array} \right]
                                               \,.
\]
Let $\tilde\phi:\osr\rightarrow\bof(\hil\oplus\hil)$ be given by $\tilde\phi(r)=\phi(r)\oplus\phi(r)$. Because $\phi(r)\psi(u_i)=\psi(u_i)\phi(r)$ for all $-n\leq i\leq n$,
functional calculus yields $\tilde\phi(r)w_i=w_i\tilde\phi(r)$ for all $r\in\osr$ and $-n\leq i\leq n$.
Since $\cstar(\fn)$ is universal, there is a unital homomorphism $\pi:\cstar(\fn)\rightarrow\bof(\hil\oplus\hil)$ such that $\pi(u_i)=w_i$ for each $1\leq i\leq n$,
whence $\pi(u_i)=w_i$ for all $-n\leq i\leq n$. Let $p=\left[\begin{array}{cc} 1_\hil & 0 \\ 0&0\end{array}\right]\in \bof(\hil\oplus\hil)$ and define
$\tilde\psi:\cstar(\fn)\rightarrow\bh$ by $\tilde\phi(x)=p\pi(x)p_{\vert\hil}$. Thus, $\tilde\psi$ is a ucp extension of $\psi$. Moreover, as the range of $\tilde\phi$
commutes with $p$ and with the range of $\pi$, $\tilde\psi$ and $\phi$ have commuting ranges. 
\end{proof}

\begin{lemma}\label{ali} The following statements are equivalent for a unital C$^*$-algebra $\csta$:
\begin{enumerate}
\item\label{ali-1} $\csta\omin\cstar(\fn)=\csta\omax\cstar(\fn)$ for some $n\geq2$;
\item\label{ali-2} $\csta\omin\cstar(\fn)=\csta\omax\cstar(\fn)$ for every $n\geq2$;
\item\label{ali-3} $\csta\omin\cstar(\fni)=\csta\omax\cstar(\fni)$.
\end{enumerate}
\end{lemma}

\begin{proof} The free group $\fni$ is a subgroup of $\ftwo$ and, hence, of $\fn$, for any fixed $n\geq2$.
Therefore, $\csta\omin\cstar(\fni)\subset\csta\omin\cstar(\fn)$ is an inclusion of C$^*$-algebras. By \cite[Proposition 8.8]{Pisier-book} there is a
ucp projection $\psi:\cstar(\fn)\rightarrow\cstar(\fn)$ with range $\cstar(\fni)$ (considered as a unital C$^*$-subalgebra
of $\cstar(\fn)$). Thus,
\[
\csta\omin\cstar(\fni)\subset\csta\omin\cstar(\fn)=\csta\omax\cstar(\fn)\rightarrow  \csta\omax\cstar(\fni)
\]
yields a ucp map $\csta\omin\cstar(\fni)\rightarrow\csta\omax\cstar(\fni)$, which implies that
$\csta\omin\cstar(\fni)=\csta\omax\cstar(\fni)$. This proves the implication \eqref{ali-1}$\Rightarrow$\eqref{ali-3}. But using the fact that any two countable
free groups are subgroups of each other, the same arguments apply to obtain the equivalence of statements \eqref{ali-1}, \eqref{ali-2}, and \eqref{ali-3}.
\end{proof}

\begin{theorem}\label{alg psn} For a fixed $n\geq 2$, the following statements are equivalent for a unital C$^*$-algebra $\csta$.
\begin{enumerate}
\item\label{ap-1} $\csta$ has $\psn$.
\item\label{ap-2} $\csta\omin\oss_{n-1}=\csta\omax\oss_{n-1}$.
\item\label{ap-3} $\csta\omin\cstar(\fnl)=\csta\omax\cstar(\fnl)$.
\end{enumerate}
If $n=2$, then the equivalent statements \eqref{ap-1}, \eqref{ap-2}, and \eqref{ap-3} hold for every $\csta$. If $n\geq3$, then the equivalent statements \eqref{ap-1}, \eqref{ap-2}, and \eqref{ap-3}
hold if and only if $\csta$ has WEP.
\end{theorem}

\begin{proof} \eqref{ap-1}$\Leftrightarrow$\eqref{ap-2} The complete quotient map $\phi:\tn\rightarrow\oss_{n-1}$
induces a complete quotient map ${\rm id}_\osr\otimes\phi:\osr\omax\tn\rightarrow\osr\omax\oss_{n-1}$ for every operator
system $\osr$ \cite[Proposition 1.6]{farenick--paulsen2011}. Such is the case in particular for $\osr=\csta$. Because $\omin=\oc$ if one of the
tensor factors is $\tn$
\cite[Proposition 4.1]{farenick--paulsen2011}
and because $\oc=\omax$ if one of the tensor factors is a unital C$^*$-algebra \cite[Theorem 6.6]{kavruk--paulsen--todorov--tomforde2011},
we deduce that $\csta\omin\tn=\csta\omax\tn$.
Now consider the following commutative diagram:
\[ \begin{CD}
\csta\omin\tn   @>{ \cong}>>\csta\omax\tn  \\
@V{{\rm id}_{\csta}\otimes\phi}VV           @VV{{\rm id}_{\csta}\otimes\phi}V  \\
\csta\omin\oss_{n-1} @>>{\theta={\rm id}}> \csta\omax\oss_{n-1}\,.
\end{CD}\]
The top arrow is a complete order isomorphism, the rightmost arrow is a complete quotient map, 
while the bottom arrow is a linear isomorphism $\theta$ in which $\theta^{-1}$ is completely positive.
By \cite[Lemma 5.1]{farenick--paulsen2011}, ${\rm id}_{\csta}\otimes\phi:\csta\omin\tn \rightarrow\csta\omin\oss_{n-1}$
is a complete quotient map 
if and only if $\theta$ is a complete order isomorphism. But since ${\rm id}_{\csta}\otimes\phi:\csta\omin\tn \rightarrow\csta\omin\oss_{n-1}$
is a complete quotient map 
if and only if $\csta$ has $\psn$ (by Theorem \ref{fp thm}(2)), we deduce the equivalence \eqref{ap-1}$\Leftrightarrow$\eqref{ap-2}.

\eqref{ap-2}$\Rightarrow$\eqref{ap-3}. By Lemma \ref{incl lemma}, $\csta\oc\oss_{n-1}\coisubset\csta\oc\cstar(\fnl)$. 
Thus, $\csta\omax\oss_{n-1}\subset\csta\omax\cstar(\fnl)$ since $\oc=\omax$
if one of the tensor factors is a C$^*$-algebra. By the hypothesis that $\csta\omin\oss_{n-1}=\csta\omax\oss_{n-1}$, the 
inclusion map $\iota:\csta\omin\oss_{n-1}\rightarrow\csta\omax\cstar(\fnl)$ is ucp. Assuming that $\csta\omax\cstar(\fnl)$
is faithfully represented as a unital C$^*$-subalgebra of $\bh$, there is a ucp extension $\tilde\iota:\csta\omin\cstar(\fnl)\rightarrow\bh$
of $\iota$. For each $-(n-1)\leq k\leq (n-1)$ and every $a\in\csta$,
\[
\begin{array}{rcl}
\tilde\iota\left( (a\otimes u_k)^*(a\otimes u_k) \right) &=& \tilde\iota\left( a^*a\otimes u_k^*u_k\right) =
\tilde\iota\left(a^*a\otimes 1\right) = \iota(a^*a\otimes 1) \\ 
&=& \iota(a\otimes u_k)^*\iota(a\otimes u_k) =\tilde\iota(a\otimes u_k)^*\tilde\iota(a\otimes u_k)\,.
\end{array}
\]
Likewise, $\tilde\iota\left( (a\otimes u_k)(a\otimes u_k)^* \right)=\tilde\iota(a\otimes u_k)\tilde\iota(a\otimes u_k)^*$.
Thus, the multiplicative domain of $\tilde\iota$ contains $\{\sum_{k=-(n-1)}^{n-1}a_k\otimes u_k\,:\,a_k\in\csta\}=\csta\otimes\oss_{n-1}$, 
and therefore it also contains the C$^*$-algebra $\csta\otimes\oss_{n-1}$ generates,
namely $\csta\omin\cstar(\fnl)$. Thus, $\tilde\iota$ is a unital homomorphism and, hence, 
$\csta\omin\cstar(\fnl)=\csta\omax\cstar(\fnl)$.

\eqref{ap-3}$\Rightarrow$\eqref{ap-2}. As before, Lemma \ref{incl lemma} yields $\csta\omax\oss_{n-1}\coisubset\csta\omax\cstar(\fnl)$.  By hypothesis,
the inclusion map $\iota:\csta\omin\oss_{n-1}\rightarrow\csta\omax\cstar(\fnl)$ is ucp with range $\csta\omax\oss_{n-1}$.
Thus, $\csta\omin\oss_{n-1}=\csta\omax\oss_{n-1}$. This completes the proof of the equivalence of statements \eqref{ap-1}--\eqref{ap-3}.

To prove the additional assertions, first let $n=2$. In this case $\cstar(\fnl)=C(\mathbb T)$. Because $\csta\omin C(\mathbb T)=\csta\omax C(\mathbb T)$
for every C$^*$-algebra $\csta$, condition \eqref{ap-3} holds for every $\csta$.

Next, assume $n\geq 3$. If $\csta$ has WEP, then $\csta$ has property $\psn$ for every $n\geq2$
\cite[Theorem 6.2]{farenick--paulsen2011}, and in particular for the fixed $n$ is the statement of the theorem.
Conversely, if $\csta$ has property $\psn$, then using the equivalence of \eqref{ap-1} and \eqref{ap-2} we have
$\csta\omin\cstar(\fnl)=\csta\omax\cstar(\fnl)$. Lemma \ref{ali} yields
$\csta\omin\cstar(\fni)=\csta\omax\cstar(\fni)$. Therefore, by Kirchberg's Theorem \cite[Proposition 1.1(iii)]{kirchberg1993}, 
$\csta$ has WEP.
\end{proof}

The following fact was noted (for a fixed $n\geq3$ rather than for $n=3$ as below)
in the proof of Theorem \ref{alg psn}, but is important enough to isolate and state separately.

\begin{corollary}\label{3 is enough} The following statements are equivalent for a unital C$^*$-algebra $\csta$:
\begin{enumerate}
\item\label{3-1} $\csta$ has property $\psthree$;
\item\label{3-2} $\csta$ has WEP;
\item\label{3-3} $\csta$ has 
property $\psn$ for every $n\geq2$.
\end{enumerate}
\end{corollary}

\begin{proof} The equivalence of \eqref{3-2} and \eqref{3-3} is proven in
  \cite[Theorem 6.2]{farenick--paulsen2011}. Clearly \eqref{3-3} implies
 \eqref{3-1}. Theorem \ref{alg psn} shows that \eqref{3-1} implies \eqref{3-2}.
\end{proof}

\section{ANDO'S THEOREM FOR C$^*$-ALGEBRAS}

Consider the following two pairs of logical assertions concerning $x_1,\dots,x_n$ 
in  a unital C$^*$-algebra $\csta$.

\subsubsection*{Strict Form of Ando's Theorem} (Case: $n=1$)
\begin{enumerate}
\item[{\rm (A1)}] for every unitary $v\in\cstb$ in every unital C$^*$-algebra $\cstb$, the element 
$x\otimes v\in \csta\omin\cstb$ satisfies $w(x\otimes v) < \frac{1}{2}$;
\item[{\rm (A2)}] there exist strictly positive $a,b\in\csta$ such that
\[
\left[ \begin{array}{cc} a & x \\ x^* & b\end{array}\right]  
\] 
is strictly positive in $\mtwo(\csta)$ and $a+b= 1$.
\end{enumerate}

The Strict Form of Ando's Theorem is the assertion that (A1) and (A2) are logically equivalent.

\subsubsection*{Multivariable Version of the Strict Form of Ando's Theorem} (Case: $n\geq2$)
\begin{enumerate}
\item[{\rm (A1')}] for every unital C$^*$-algebra $\cstb$ and unitaries $v_1,\dots, v_n\in\cstb$, the numerical radius of
$\sum_{j=1}^nx_j\otimes v_j$ in $\csta\omin\cstb$ satisfies
\[
w(x_1\otimes v_1\,+\,\cdots\,+\,x_n\otimes v_n) \,<\, \frac{1}{2}\,;
\]
\item[{\rm (A2')}] there exist strictly positive $a_1,\dots,a_{n+1}\in\csta$ such that
\[
\left[ \begin{array}{ccccc} a_1 & x_1 & 0 & \cdots & 0 \\ x_1^* & a_2 & x_2 & & \vdots \\
0& x_2^* & \ddots & \ddots & 0 \\ \vdots & & \ddots & a_{n} & x_{n} \\ 
0 & \cdots & 0 & x_{n}^* & a_{n+1} \end{array}
\right]
\]
is a strictly positive element of $\mn(\csta)$ and $a_1+\cdots+a_{n+1}=1$.
\end{enumerate}

The Multivariable Version of the Strict Form of Ando's Theorem is the assertion that (A1') and (A2') are logically equivalent.

\begin{theorem}\label{interesting} {(\rm Ando's Theorem, Nuclearity, and WEP)}
\begin{enumerate}  
\item The Strict Form of Ando's Theorem holds for every 
C$^*$-algebra $\csta$ if and only if $C(\mathbb T)$ is a nuclear C$^*$-algebra.
\item The Multivariable Version of the Strict Form of Ando's Theorem holds for a C$^*$-algebra $\csta$ if and only if $\csta$ has WEP.
\end{enumerate}
\end{theorem}

\begin{proof}
An element $z$ of a unital C$^*$-algebra $B$ satisfies $w(z)<\frac{1}{2}$ if and only if $1+2\Re(z)$ is strictly positive. Because
$C(\mathbb T)$ is the universal C$^*$-algebra generated by a unitary,
Theorems \ref{alg psn} and \ref{lifting to matrices} complete the proof.
\end{proof}

Because $C(\mathbb T)$ is (by a fairly simple argument) known to be nuclear, Theorem \ref{interesting}
explains why the strict form of Ando's theorem holds for every unital
C$^*$-algebra $\csta$.

Once one knows that the strict form of Ando's theorem holds, the
remaining equivalences of Theorem~1.1 are easily verified.

\begin{remark}\label{Bunce} It is also possible to use the proof of Ando's theorem
  given by Bunce \cite{Bu} to prove the strict form of Ando's
  theorem. Bunce's proof has the added benefit of giving us a concrete
  formula for the entries $a$ and $b$ satisfying (A2). Bunce's proof
  shows that if we let $Q=(q_{i,j})_{i,j \in \mathbb N}$ be the infinite matrix with entries from
  $\csta$ given by $q_{i,i}= 1, q_{i,i+1} = x, q_{i+1,i} = x^*$ and
  $q_{i,j} =0$ for all other pairs, then one may take $a= S_0(Q)$
  where $S_0(Q)$ is the ``short'', in the language of Anderson and
  Trapp \cite{AT}, of the operator $Q$ to the first
  entry. It is not difficult to compute that in this case
\[a= S_0(Q) = 1 - x^*(Q^{-1})_{1,1}x,\]
where $(Q^{-1})_{1,1}$ denotes the $(1,1)$-entry of the matrix
defining the inverse of $Q.$ Finally, since we are assuming that $w(x)
< 1/2$ the matrix $Q$ is strictly positive and tri-diagonal so that
the entries of $Q^{-1}$ can be seen to be in the unital C*-algebra
generated by $x$ and $x^*.$ 
\end{remark}

Combining Remark \ref{Bunce} with Theorem~\ref{interesting} leads to a
rather perverse proof that $C(\mathbb T)$ is a nuclear C$^*$-algebra.

\section{C$^*$-SUBALGEBRAS OF $\bh$ WITH WEP}

Corollary \ref{3 is enough}
shows that a unital C$^*$-algebra $\csta$ has WEP if and only if $\csta$ has property $\psthree$, while 
Theorem \ref{lifting to matrices} shows that $\csta$ has $\psthree$ if and only if strictly positive elements of 
$\csta\omin\oss_2$ lift to strictly positive elements of $\csta\omin\tthree$.
Therefore, these results yield some new characterisations of WEP,
which we explore in this section.

\begin{theorem}\label{wep3x3-1} The following statements are equivalent for a unital C$^*$-subalgebra $\csta\subset\bh$:
\begin{enumerate}
\item\label{w3-1} $\csta$ has WEP;
\item\label{w3-2} whenever, for arbitrary $p\in\mathbb N$,
there exist $X_1,X_2\in\mpee(\csta)$ and $A,B,C\in\mpee(\bh)$ such that $A+B+C=I$
and 
\begin{equation}\label{3x3m}
\left[ \begin{array}{ccc} A& X_1 & 0 \\ X_1^* & B & X_2\\ 0& X_2^*& C\end{array}\right]    \tag{6.1.1}
\end{equation}
is strictly positive in $\mathcal M_{3p}(\bh)$, there also exist $\tilde A,\tilde B,\tilde C\in \mpee(\csta)$ with the same property.
\end{enumerate}
\end{theorem}

As noted in the Introduction, one important aspect of Theorem \ref{wep3x3-1} is that this characterisation of WEP
requires a specific property of just one of the many faithful
representations that an abstract C$^*$-algebra $\csta$ can
have. Consequently, if one faithful representation has this property,
then all do.
In constrast, Lance's definition of WEP requires that every faithful representation of $\csta$ satisfy a certain property 
(namely, that there exist a weak expectation into the double
commutant) or equivalently, that a special representation, namely the
universal representation have this property. In further contrast, the
reduced atomic representation of a C$^*$-algebra always possesses a
weak expectation into the double commutant. So one faithful
representation possessing a weak expectation is not enough to
characterise WEP.

Theorem \ref{wep3x3-1} is also a new characterisation of injectivity for von~Neumann algebras.

\begin{corollary}\label{inj1} If $\csta\subset\bh$ is a von Neumann algebra, then 
$\csta$ is injective if and only if statement \eqref{w3-2} of Theorem \ref{wep3x3-1} holds for $\csta$.
\end{corollary}

\begin{proof} A von Neumann algebra has WEP if and only if it is injective.
\end{proof}

We do not know if in the results above it is sufficient to consider
only the case $p=1$.

To conclude this section, we present a second characterisation of WEP in terms of strict row contractions.
 
\begin{lemma}\label{cholesky} For any unital C$^*$-algebra $\csta$, the matrix
\[
\left[ \begin{array}{ccc} 1& x_1 &0 \\ x_1^*&1&x_2 \\ 0&x_2^*&1 \end{array}\right]
\]
is strictly positive if and only if $1-x_1^*x_1-x_2x_2^*$ is strictly positive.
\end{lemma}

\begin{proof} Let $y$ denote the $3\times3$ matrix in question and factor $y$ as
\[
y\,=\,\left[ \begin{array}{ccc} 1& 0 &0 \\ x_1^*&1&0 \\ 0&0&1 \end{array}\right]
\,
\left[ \begin{array}{ccc} 1& 0 &0 \\ 0&1-x_1^*x_1&x_2 \\ 0&x_2^*&1 \end{array}\right]
\,\left[ \begin{array}{ccc} 1& x_1 &0 \\0&1&0 \\ 0&0&1 \end{array}\right]
\,.
\]
Thus, $y$ is strictly positive if and only if the middle factor is, which in turn is strictly positive if and only if 
$\left[ \begin{array}{cc} 1-x_1^*x_1&x_2 \\ x_2^*&1 \end{array}\right]$ is. But in $\mtwo(\csta)$ 
this matrix is unitarily equivalent to
\[
\left[ \begin{array}{cc} 1 &0 \\ x_2&1 \end{array}\right] \,
\left[ \begin{array}{cc} 1 &0 \\ 0&1-x_1^*x_1-x_2x_2^* \end{array}\right]\,
\left[ \begin{array}{cc} 1 &x_2^* \\ 0&1 \end{array}\right] \,,
\]
which is strictly positive if and only if $1-x_1^*x_1-x_2x_2^*$ is strictly positive.
\end{proof}

For $X_1,X_2\in \bh$, the condition that $I-X_1^*X_1-X_2X_2^*$ be strictly positive
is equivalent to the condition that $(X_1^*,X_2):\hil\oplus\hil\rightarrow\hil$, where $(X_1^*,X_2)(\xi\oplus\eta)=X_1^*\xi+X_2\eta$, be a strict
(row) contraction, namely
\[
\left\| (X_1^*,X_2) \right\|\,<\,1.
\]

\begin{theorem}\label{wep3x3-2} The following statements are equivalent for a unital C$^*$-subalgebra $\csta\subset\bh$:
\begin{enumerate}
\item\label{w-1} $\csta$ has WEP;
\item\label{w-2} whenever, for arbitrary $p \in \mathbb N,$ $X_1,X_2\in \mpee(\csta)$ are operators for which there exist strictly positive $A,B,C\in \mpee(\bh) $ such that $A+B+C=I$
and 
\[
\left\| \left(B^{-1/2}X_1A^{-1/2}, \, B^{-1/2}X_2C^{-1/2} \right) \right\|\,<\,1,  
\]
there also exist strictly positive $\tilde A,\tilde B,\tilde C\in \mpee(\csta)$ with the same property.
\end{enumerate}
\end{theorem}

\begin{proof} 
Suppose that $X_1,X_2\in \mpee(\csta)$ and that  $A,B,C\in \mpee(\bh) $ are strictly positive
operators such that $A+B+C=I$. Let $Y$ denote the  $3\times 3$ matrix \eqref{3x3m}
and let $D=A\oplus B\oplus C$, which is strictly positive. Observe that $Y$
is strictly positive if and only if 
$D^{-1/2}YD^{-1/2}$ is strictly positive, which by Lemma \ref{cholesky} occurs precisely when 
$(B^{-1/2}X_1^*A^{-1/2}, B^{-1/2}X_2C^{-1/2})$ is a strict row contraction.
Relabeling $X_1$ by $X_1^*$
yields the result.
\end{proof} 

\begin{corollary}\label{inj2} If $\csta\subset\bh$ is a von Neumann algebra, then 
$\csta$ is injective if and only if statement (2) of Theorem \ref{wep3x3-2} holds for $\csta$.
\end{corollary}

Corollary \ref{inj2} shows that, in Theorem \ref{wep3x3-2}, if a unital C$^*$-subalgebra $\csta\subset\bh$
has WEP, then one should not expect to replace the original $A,B,C\in\bh$ with operators $A,B,C\in \csta''$
if the von Neumann algebra $\csta''$ is non-injective.

 \section{THE CONNES EMBEDDING PROBLEM}
 
 Perhaps the most outstanding open problem in operator algebra theory at
 present is Connes' Embedding Problem: is every II${}_1$-factor with separable predual
 a subfactor of the ultrapower $R^\omega$ of the hyperfinite II${}_1$-factor $R$?
 Kirchberg's equivalent formulation of Connes' Embedding Problem is the problem of 
 whether $\cstar(\fni)$ has WEP \cite[Theorem 13.3.1]{Brown--Ozawa-book}, \cite[Proposition 8.1]{kirchberg1993}. 
 Below we state a new equivalent
 form of the problem.
 
 \begin{theorem}\label{cep} The following statements are equivalent:
 \begin{enumerate}
 \item The Connes Embedding Problem has an affirmative solution;
 \item for every $p\in\mathbb N$ and every $x_1,x_2\in\mpee(\cstar(\ftwo))$ for which
\[
1\otimes 1 \,+\,(x_1\otimes u_1)\,+\,(x_2\otimes u_2)\,+\,(x_1^*\otimes u_1^*)\,+\,(x_2^*\otimes u_2^*)
\]
is strictly positive in $\mpee(\cstar(\ftwo))\omin\cstar(\ftwo)$, there are strictly positive $a,b,c\in\mpee(\cstar(\ftwo))$ such that
$a+b+c=1\in\mpee(\cstar(\ftwo))$ and
\[
\left[ \begin{array}{ccc} a&x_1 & 0 \\ x_1^* & b & x_2\\ 0& x_2^*& c\end{array}\right]
\]
is strictly positive in $\mathcal M_{3p}(\cstar(\ftwo))$.
 \end{enumerate}
 \end{theorem}

\begin{proof} Statement (2) is equivalent to the assertion that $\cstar(\ftwo)$ has property $\psthree$, which in turn
is equivalent to the assertion that $\cstar(\ftwo)$ has WEP. But $\cstar(\ftwo)$ has WEP if and only if
$\cstar(\ftwo)\omin\cstar(\fni)=\cstar(\ftwo)\omax\cstar(\fni)$ \cite{kirchberg1993}. By Lemma \ref{ali}, 
$\cstar(\ftwo)\omin\cstar(\fni)=\cstar(\ftwo)\omax\cstar(\fni)$ if and only if 
$\cstar(\fni)\omin\cstar(\fni)=\cstar(\fni)\omax\cstar(\fni)$, which is equivalent to the assertion that
the Connes Embedding Problem has an affirmative solution \cite{kirchberg1993}. 
\end{proof}

Alternatively, if one fixes a faithful, unital representation of $\cstar(\ftwo)$
on some Hilbert space, then one may also use Theorem~\ref{wep3x3-1} or
Theorem~\ref{wep3x3-2} to give other equivalences of Connes' Embedding
Problem. We state these below.

\begin{theorem} The following statements are equivalent for a fixed faithful unital representation of $\cstar(\ftwo)$ on
some Hilbert space $\hil$:
\begin{enumerate}
\item The Connes Embedding Problem has an affirmative solution;
\item for arbitrary $p \in \mathbb N,$  whenever $X_1,X_2\in \mpee(\cstar(\ftwo))$ 
are operators for which there exist strictly positive $A,B,C\in \mpee(\bh) $ such that $A+B+C=I$
and 
\[
\left[ \begin{array}{ccc} A& X_1 & 0 \\ X_1^* & B & X_2\\ 0& X_2^*& C\end{array}\right]
\]
is strictly positive in $\mathcal M_{3p}(\bh)$, there also exist
$\tilde A,\tilde B,\tilde C\in \mpee(\cstar(\ftwo))$ with the same property;
\item  for arbitrary $p \in \mathbb N,$  whenever $X_1,X_2\in \mpee(\cstar(\ftwo))$ 
are operators for which there exist strictly positive $A,B,C\in \mpee(\bh) $ such that $A+B+C=I$
and 
\[
\left\| \left(B^{-1/2}X_1A^{-1/2}, \, B^{-1/2}X_2C^{-1/2} \right) \right\|\,<\,1,
\]
then there also exist strictly positive $\tilde A,\tilde B,\tilde C\in
\mpee(\cstar(\ftwo))$ with the same property.
\end{enumerate}
\end{theorem}

In the special case of $\cstar(\ftwo)$, we also do not know if
it is sufficient to check the properties above in the case $p=1$.

\section*{Acknowledgements} 
The work of the first and third authors is supported in part by NSERC and NSF, respectively.


\begin{thebibliography}{99}
 
 \bibitem{AT} \textsc{W.N. Anderson, G.E. Trapp}, {Shorted operators, II,}
  \textit{SIAM J. Appl. Math.} \textbf{28} (1975), 60--71.

\bibitem{ando1973}
\textsc{T. Ando}, {Structure of operators with numerical radius one}, \textit{Acta Sci.
  Math. (Szeged)} \textbf{34} (1973), 11--15. 

\bibitem{Brown--Ozawa-book}
\textsc{N.P. Brown, N. Ozawa}, \textit{{$C^*$}-Algebras and Finite-Dimensional
  Approximations}, Graduate Studies in Mathematics, vol.~88, American
  Mathematical Society, Providence, RI, 2008. 


\bibitem{Bu} \textsc{J.W. Bunce}, {Shorted operators and the structure of
    operators with numerical radius one,} \textit{Integral Equations 
  Operator Theory} \textbf{11} (1988), 287--291.

\bibitem{choi--effros1977}
\textsc{M.D. Choi, E.G. Effros}, {Injectivity and operator spaces}, \textit{J.
  Functional Analysis} \textbf{24} (1977), 156--209.  
  
\bibitem{farenick--paulsen2011}
\textsc{D. Farenick, V.I. Paulsen}, {Operator system quotients of matrix
  algebras and their tensor products},  
\textit{Math. Scand.}, to appear. 

\bibitem{kavruk}
\textsc{A.S. Kavruk}, {Nuclearity in operator systems and applications}, 
\textit{preprint}, 2011 (arXiv:OA/1008.281).

\bibitem{kavruk--paulsen--todorov--tomforde2011}
\textsc{A.S. Kavruk, V.I. Paulsen, I.G. Todorov, M. Tomforde},
 {Tensor products of operator systems}, \textit{J.
 Functional Analysis} \textbf{261} (2011), 267--299.
 
 
\bibitem{kavruk--paulsen--todorov--tomforde2010}
\textsc{A.S. Kavruk, V.I. Paulsen, I.G. Todorov, M. Tomforde}, {Quotients,
  exactness and nuclearity in the operator system category}, \textit{preprint}, 2010 (arXiv:OA/1008.281).
   
\bibitem{kirchberg1993}
\textsc{E. Kirchberg}, {On nonsemisplit extensions, tensor products and exactness
  of group {$C^*$}-algebras}, \textit{Invent. Math.} \textbf{112} (1993), no.~3,
  449--489. 
  
\bibitem{Paulsen-book}
\textsc{V. Paulsen}, \textit{Completely Bounded Maps and Operator Algebras}, Cambridge
  Studies in Advanced Mathematics, vol.~78, Cambridge University Press,
  Cambridge, 2002.  
  
\bibitem{Pisier-book}
\textsc{G. Pisier}, \textit{Introduction to Operator Space Theory}, London Mathematical
  Society Lecture Note Series, vol. 294, Cambridge University Press, Cambridge,
  2003.  



\end{thebibliography}
\end{document}